\documentclass{amsart}
\usepackage{a4wide}

\newcommand{\cl}{\operatorname{cl}}
\newcommand{\ka}{\kappa}

\newtheorem{problem}{Problem}
\newtheorem{theorem}{Theorem}
\newtheorem{proposition}{Proposition}

\newtheorem{corollary}{Corollary}

\newtheorem{lemma}{Lemma}

\newtheorem{remark}{Remark}

\title[On regular $\kappa$-bounded spaces admitting only constant continuous mappings into...]{On regular $\kappa$-bounded spaces admitting only constant continuous mappings into $T_1$ spaces of pseudo-character $\leq \kappa$}
\author[S. Bardyla]{Serhii Bardyla}
\address{S.Bardyla: University of Vienna, Austria}
\thanks{The work of the first author is supported by the Austrian Science Fund FWF (Grant  I
3709 N35).}
\email{sbardyla@yahoo.com}
\author[A.~Osipov]{Alexander Osipov}
\address{A. Osipov: Krasovskii Institute of Mathematics and Mechanics, Ural Federal University, and Ural State
University of Economics, Yekaterinburg, Russia}
\email{oab@list.ru}

\begin{document}
\begin{abstract}
In this paper for each cardinal $\kappa$ we construct an infinite $\kappa$-bounded (and hence countably compact) regular space $R_{\kappa}$ such that for any $T_1$ space $Y$ of pseudo-character $\leq\kappa$, each continuous function $f:R_{\kappa}\rightarrow Y$ is constant. This result resolves two problems posted by Tzannes in Open Problems from Topology Proceedings~\cite{Prob} and extends results of Ciesielski and Wojciechowski~\cite{CW} and Herrlich~\cite{H}.
\end{abstract}
\maketitle

We shall follow the terminology of~\cite{Eng, Kun}. Throughout of this paper all cardinals are assumed to be infinite.

Regular spaces on which every continuous real-valued function (or, more generally, spaces on which every continuous function into a given space $Y$) is constant are of particular interest in general topology. Such spaces were constructed and investigated in~\cite{1, 2, CW, 3, 3.5, H, 4, 5, KP, 6, 7, Tz, 8}. For instance, a well-known result of Herrlich~\cite{H} states the following:
\begin{theorem}[{\cite[Theorem]{H}}]\label{Her}
Let $Y$ be a topological space. The following conditions are equivalent:
\begin{itemize}
\item $Y$ is a $T_1$-space;
\item there exists a regular space $X$ (having at least two points), such that every
continuous map from $X$ to $Y$ is constant.
\end{itemize}
\end{theorem}

Also, Ciesielski and Wojciechowski in~\cite{CW} proved the following:
\begin{theorem}[{\cite[Theorem 7]{CW}}]\label{CW}
For any uncountable cardinal $\kappa$ there exists a regular space $Y$ of cardinality $\kappa$ such that any continuous function from $Y$ into any Hausdorff space $Z$ with a countable pseudo-character is constant.
\end{theorem}

However, all known examples of regular spaces on which every continuous real-valued function is constant are far from being countably compact. In~\cite{Tz} Tzannes constructed a Hausdorff countably compact space $T$ on which every continuous real-valued function is constant. Nevertheless, the space $T$ is strongly non-Urysohn. In particular, no pair of distinct points of $T$ have disjoint closed neighborhoods.
In~\cite{Prob} Tzannes posed the following two problems:

\begin{problem}[{\cite[Problem C65]{Prob}}]\label{p1}
Does there exist a regular (first countable, separable)
countably compact space on which every continuous real-valued function is constant?
\end{problem}

\begin{problem}[{\cite[Problem C66]{Prob}}]\label{p2}
Does there exist, for every Hausdorff space $R$, a regular
(first countable, separable) countably compact space on which every continuous
function into $R$ is constant?
\end{problem}

Let $\kappa$ be a cardinal. A topological space $X$ is called {\em $\kappa$-bounded} if the the closure of each subset $A\subset X$ of cardinality $\leq \kappa$ is compact. It is clear that each $\kappa$-bounded space is countably compact and each $\kappa$-bounded space of density $\leq \kappa$ is compact.

The pseudo-character $\psi(X)$ of a space $X$ is the smallest cardinal $\lambda$ such that each point is the intersection of a family of cardinality $\leq\lambda$ of sets which are open in $X$.

In this paper for each cardinal $\kappa$ we construct an infinite $\kappa$-bounded regular space $R_{\kappa}$ such that for each $T_1$ space $Y$ of pseudo-character $\leq\kappa$, each continuous function $f:R_{\kappa}\rightarrow Y$ is constant. This result resolves Problems~\ref{p1},~\ref{p2} and extends Theorems~\ref{Her},~\ref{CW}.

\section{Wallman $\kappa$-bounded extension}
For a subset $A$ of a topological space $X$ by $\cl_{X}(A)$ (or simply $\overline{A}$) we denote the closure of $A$ in $X$.

We recall \cite[\S3.6]{Eng} that the Wallman extension $W(X)$ of a $T_1$ space $X$ consists of closed ultrafilters, i.e., families $\mathcal{F}$ of closed subsets of $X$ satisfying the following conditions:
\begin{itemize}
\item $\emptyset\notin\mathcal{F}$;
\item $A\cap B\in\mathcal{F}$ for any $A,B\in\mathcal{F}$;
\item a closed set $F\subset X$ belongs to $\mathcal F$ if $F\cap A\ne\emptyset$ for every $A\in\mathcal{F}$.
\end{itemize}
For any $A\subset X$ put $$\langle A\rangle=\{\mathcal F\in W(X)\mid \hbox{ there exists } F\in\mathcal F \hbox{ such that }F\subset A\}.$$
The Wallman extension $W(X)$ of $X$ carries the topology generated by the base consisting of the sets
$\langle U\rangle$ where $U$ runs over open subsets of $X$.

By Theorem~\cite[3.6.21]{Eng}, the Wallman extension $W(X)$ is $T_1$ and compact. By Theorem~\cite[3.6.22]{Eng} a $T_1$-space $X$ is normal if and only if $W(X)$ is Hausdorff.

Consider the map $j_X:X\to W(X)$ assigning to each point $x\in X$ the principal closed ultrafilter consisting of all closed sets $F\subset X$ containing the point $x$. It is clear that the image $j_X(X)$ is dense in $W(X)$. Since the space $X$ is $T_1$, Theorem~3.6.21 from~\cite{Eng} provides that the map $j_X:X\to W(X)$ is a topological embedding. So, $X$ can be identified with the subspace $j_X(X)$ of $W(X)$.


\begin{lemma}\label{W}
Let $\mathcal{F}\in W(X)$ and $A\subset X$. Then $\mathcal{F}\in \cl_{W(X)}(A)$ if and only if $\cl_X(A)\in\mathcal{F}$.
\end{lemma}

\begin{proof}
$(\Rightarrow)$ To derive a contradiction assume that $\mathcal{F}\in \cl_{W(X)}(A)$, but $\cl_X(A)\notin \mathcal{F}$. Using the maximality of $\mathcal{F}$ we can find an element $F\in \mathcal{F}$ such that $F\subset X\setminus \cl_X(A)$.
Then $\langle X\setminus \cl_X(A)\rangle$ is an open neighborhood of $\mathcal{F}$ such that $\langle X\setminus \cl_X(A)\rangle\cap A=\emptyset$ which implies a contradiction.

$(\Leftarrow)$ Assume that $\cl_X(A)\in \mathcal{F}$ and fix any open neighborhood $\langle U\rangle$ of $\mathcal{F}$. Then there exists an element $H\in \mathcal{F}$ such that $H\subset U$. Observe that $\cl_X(A)\cap H\in\mathcal{F}$ and $\cl_X(A)\cap H\subset U$. Hence $A\cap \langle U\rangle\neq \emptyset$ witnessing that $\mathcal{F} \in \cl_{W(X)}(A)$.
\end{proof}

Given a cardinal $\kappa$, in the Wallman extension $W(X)$ of a $T_1$-space $X$, consider the subspace
$$W_{\kappa}X={\textstyle\bigcup}\{\cl_{W(X)}(C)\mid C\subset X \hbox{ and }|C|\le\kappa\}$$ of $W(X)$. The space $W_{\kappa}(X)$ is called the {\em Wallman $\kappa$-bounded extension} of $X$. The Wallman $\kappa$-bounded extension was introduced and investigated in~\cite{BBRk}. In particular,
there it was proved the following:
\begin{proposition}[{\cite[Proposition 3.2]{BBRk}}]\label{BBR}
For any $T_1$ space $X$, the space $W_{\kappa}(X)$ is $\kappa$-bounded.
\end{proposition}

For a topological space $X$ by $2^X$ we denote the family of all closed subsets of $X$.
Let $\mathcal{C}$ be any subfamily of $2^X$.
Following~\cite{BBRk}, a topological space $X$ is called {\em totally $\mathcal{C}$-normal} if for any disjoint sets $A\in\mathcal{C}$ and $B\in 2^X$ there exist disjoint open sets $U,V\subset X$ such that $A\subset U$ and $B\subset V$.

Let $\kappa$ be a cardinal. A topological space $X$ is called {\em totally $\overline{\kappa}$-normal} if it is totally $\mathcal C$-normal for the family $\mathcal C$ of closed subsets of the closures of subsets of cardinality $\le\kappa$ in $X$.

\begin{proposition}[{\cite[Proposition 2.9]{BBRk}}]\label{BBR1}
Each $\kappa$-bounded regular space $X$ is totally $\overline{\kappa}$-normal.
\end{proposition}

\begin{proposition}\label{Wk}
The Wallman $\kappa$-bounded extension $W_{\kappa}(X)$ of $X$ is regular iff $X$ is totally $\overline{\kappa}$-normal.
\end{proposition}

\begin{proof}
($\Rightarrow$) Assume that the space $W_{\kappa}(X)$ is regular. By Proposition~\ref{BBR}, the space $W_{\kappa}(X)$ is $\kappa$-bounded.
To show that the space $X$ is totally $\overline{\kappa}$-normal, take any subset $C\subset X$ of cardinality $|C|\le\kappa$ and two disjoint closed subsets $F,E$ of $X$ such that $F\subset \cl_{X}(C)$. Lemma~\ref{W} implies that $\cl_{W_{\kappa}(X)}(F)\cap \cl_{W_{\kappa}(X)}(E)=\emptyset$.
Since $\cl_{W_{\kappa}(X)}(F)\subset \cl_{W_{\kappa}(X)}(C)$ and $|C|\leq \kappa$, the set $\cl_{W_{\kappa}(X)}(F)$ is compact. Since $W_{\kappa}(X)$ is regular the sets $\cl_{W_{\kappa}(X)}(F)$ and $\cl_{W_{\kappa}(X)}(E)$ have disjoint open neighborhoods $U_1$ and $U_2$, respectively, in $W_{\kappa}(X)$. Put $V_1=U_1\cap X$ and $V_2=U_2\cap X$. Then $V_1$ and $V_2$ are disjoint open neighborhoods (in $X$) of $F$ and $E$, respectively. Hence $X$ is totally $\overline{\kappa}$-normal.

($\Leftarrow$) Assume that $X$ is totally $\overline{\kappa}$-normal. Given any closed ultrafilter $\mathcal{F}\in W_{\kappa}(X)$ and open neighborhood $\langle U\rangle$ of $\mathcal{F}$ find a closed set $F\in \mathcal{F}$ such that $F\subset U$. With no loss of generality we can assume that there exists a subset $C\subset X$ such that $|C|\le\kappa$ and  $F\subset \cl_{X}(C)$.
By the total $\overline{\kappa}$-normality of $X$, there exists an open subset $V$ of $X$ such that $F\subset V\subset \cl_X(V)\subset U$. Then Lemma~\ref{W} implies that
$$\mathcal{F}\in\langle V\rangle\subset \cl_{W_{\kappa}(X)}(\langle V\rangle)=\langle \cl_X(V)\rangle\subset \langle U\rangle$$
 witnessing that the space $W_{\kappa}(X)$ is regular.
\end{proof}

\section{Herrlich extension}
In this section we briefly recall a part of famous construction due to Herrlich~\cite{H} and establish a few important properties of it.

Let $X$ be a topological space and $\rho$ be an equivalence relation on $X$. Then for any $x\in X$ by $[x]$ we denote the equivalence class of the relation $\rho$ which contains $x$ and for any subset $A\subset X$ put $[A]=\{[x]\mid x\in A\}$. Also, we agree to denote $[x]$ ($[A]$, resp.) simply as $x$ ($A$, resp.) if $[x]=\{x\}$ ($[x]=\{x\}$ for each $x\in A$, resp.).

For any distinct elements $a,b$ of a topological space $X$ by $\operatorname{Const}(X)_{a,b}$ we denote the class of $T_1$ spaces such that $f(a)=f(b)$ for any continuous map $f:X\rightarrow Y\in \operatorname{Const}(X)_{a,b}$.
Following~\cite{H}, for any distinct points $a,b$ of a space $X$ it can be constructed a space $H_{a,b}(X)$ such that for any $Y\in \operatorname{Const}(X)_{a,b}$, each continuous function $h:H_{a,b}(X)\rightarrow Y$ is constant. It can be done in two steps.

Step 1. For each topological space $Z$ by $P(Z)$ we denote the set $Z{\times} X$ endowed with the topology $\tau$ which satisfies the following conditions:
\begin{itemize}
\item if $(z,x)\in U\in\tau$, then there exists an open set $V\subset X$ such that $x\in V$ and $\{z\}{\times}V\subset U$;
\item if $(z,a)\in U\in\tau$, then there exists an open set $W\subset Z$ such that $z\in W$ and $W{\times} \{a\}\subset U$.
\end{itemize}
By $X(Z)$ we denote the quotient space $P(Z)/\rho$ where $\rho$ is the smallest equivalence relation satisfying $(z_1,b)\rho (z_2,b)$ for any $z_1,z_2\in Z$. Observe that for any $z\in Z$ the vertical fiber $[\{z\}{\times}X]\subset X(Z)$ is homeomorphic to $X$. Since the map $h(z)=(z,a)$ is a canonical embedding of $Z$ into $X(Z)$, we can identify $Z$ with the subspace $Z{\times}\{a\}\subset X(Z)$. Observe that for each $Y\in \operatorname{Const}(X)_{a,b}$, for any continuous map $f:X(Z)\rightarrow Y$, $f((z,a))=f([z,b])$ where $z$ is an arbitrary element of $Z$.
Hence for any $Y\in \operatorname{Const}(X)_{a,b}$ each continuous map $f:X(Z)\rightarrow Y$ is constant on $Z$.

Step 2. Put $H_1=X$ and for each $n\in\mathbb{N}$ let $H_{n+1}=X(H_n)$. For any $n\in \mathbb{N}$ by $b_n$ we denote the element $[(z,b)]\in H_n$ where $z$ is any element of $H_{n-1}$. Recall that we identify $H_n$ with a subspace $H_n{\times}\{a\}$ of $H_{n+1}$ which implies that $H_n\subset H_{n+1}$ for each $n\in\mathbb{N}$. Finally, by $H_{a,b}(X)$ we denote the set $\cup_{n\in\mathbb{N}}H_n$ endowed with the topology $\tau$ which satisfies the following condition: a subset $U\subset H_{a,b}(X)$ is open in $(H_{a,b}(X),\tau)$ iff the set $U\cap H_n$ is open in $H_n$ for each $n\in \mathbb{N}$. The space $H_{a,b}(X)$ is called a {\em Herrlich extension} of $X$. Fix any $Y\in \operatorname{Const}(X)_{a,b}$. To see that each continuous function $h:H_{a,b}(X)\rightarrow Y$ is constant take any distinct points $x,y\in H_{a,b}(X)$ and observe that there exists $n\in\mathbb{N}$ such that $\{x,y\}\subset H_n$. Recall that any continuous function $g:H_{n+1}\rightarrow Y$ is constant on $H_n$. Hence the restriction of $h$ on the set $H_{n+1}$ is constant on $H_n$ witnessing that $h(x)=h(y)$.


\begin{remark}\label{rem}
The definition of the topology on $H_{a,b}(X)$ (see also~\cite{H}) implies the following:
\begin{itemize}
\item $H_n$ is a closed subset of $H_{a,b}(X)$ for each $n\in\mathbb{N}$;
\item a subset $A\subset H_{a,b}(X)$ is closed iff $A\cap H_n$ is closed in $H_n$ for each $n\in\mathbb{N}$;
\item if $X$ is regular, then so is $H_{a,b}(X)$.
\end{itemize}
\end{remark}

Let $\ka$ be a cardinal. A space $X$ is called {\em $\ka$-accumulative} if $|\overline{A}|\leq \ka$ for each subset $A\subset X$ of cardinality $\leq\ka$.

\begin{proposition}\label{prH}
For a cardinal $\kappa$ the following statements hold:
\begin{itemize}
\item [1)] if $X$ is $\kappa$-accumulative, then so is $H_{a,b}(X)$;
\item [2)] if $X$ is $\kappa$-accumulative and totally $\overline{\kappa}$-normal, then $H_{a,b}(X)$ is totally $\overline{\kappa}$-normal.
\end{itemize}
\end{proposition}

\begin{proof}
Consider statement 1. Observe that $H_1=X$ is $\kappa$-accumulative by the assumption. Assume that $H_{k-1}$ is $\kappa$-accumulative and consider a subset $A\subset H_k$ of cardinality $\leq \kappa$. For each $h\in H_{k-1}$ by $X_h$ we denote the vertical fiber $\{h\}{\times}(X\setminus\{b\})\cup\{b_n\}\subset H_k$ which is homeomorphic to $X$ and hence it is $\kappa$-accumulative. Put $A_1=\{h\in H_{k-1}\mid X_h\cap A\setminus \{b_k\}\neq\emptyset\}$. Since $|A|\leq \kappa$ the set $A_1$ has cardinality $\leq\kappa$ as well.
Put $A_2=\cup\{\cl_{X_h}(A\cap X_h)\mid h\in A_1\}$. Since for each $h\in H_{n-1}$ the subspace $X_h$ is $\kappa$-accumulative, the set $A_2$ has cardinality $\leq \kappa$. Let $A_3=\cl_{H_{k-1}}(A_1)$. Since $H_{k-1}$ is $\kappa$-accumulative and $|A_1|\leq \kappa$, the set $A_3$ has cardinality $\leq \kappa$. Finally, the definition of the topology on $H_k$ implies that
$\cl_{H_k}(A)\subset A_2\cup A_3\cup\{b_k\}$ witnessing that $H_k$ is $\kappa$-accumulative. Hence for each $n\in\mathbb{N}$ the space $H_n$ is $\kappa$-accumulative.

To see that $H_{a,b}(X)$ is $\kappa$-accumulative fix any subset $B\subset H_{a,b}(X)$ of cardinality $\leq \kappa$. By Remark~\ref{rem}, $\cl_{H_{a,b}(X)}(B)=\cup_{n\in\mathbb{N}}\cl_{H_n}(B\cap H_n)$. Since each $H_n$ is $\kappa$-accumulative $|\cl_{H_n}(B\cap H_n)|\leq \kappa$. Hence $|\cl_{H_{a,b}(X)}(B)|\leq \kappa$ witnessing that $H_{a,b}(X)$ is $\kappa$-accumulative.

Consider statement 2. By statement 1, the space $H_{a,b}(X)$ is $\kappa$-accumulative. Then it is clear that for each $n\in\mathbb{N}$ the space $H_n$ is $\kappa$-accumulative as well. Observe that $H_1=X$ is totally $\overline{\kappa}$-normal by the assumption. Assume that $H_{k-1}$ is totally $\overline{\kappa}$-normal and consider a closed subset $A\subset H_k$ of cardinality $\leq \kappa$. Fix any open set $U\subset H_k$ such that $A\subset U$. We will show that there exists an open set $W\subset H_k$ such that $A\subset W\subset \cl_{H_k}(W)\subset U$ which would provide that $H_k$ is totally $\overline{\kappa}$-normal. Since $H_{k-1}$ is totally $\overline{\kappa}$-normal there exists an open subset $V$ of $H_{k-1}$ such that $A\cap H_{k-1}\subset V\subset \cl_{H_{k-1}}(V)\subset U\cap H_{k-1}$.
Moreover, since $U$ is open in $H_k$ for each $h\in V$ there exists an open neighborhood $V_h(a)$ of $a$ in $X_h$ such that $\{h\}{\times} V_h(a)\subset U$. Since $X_h$ is regular we can assume that $\cl_{X_h}(V_h(a))\subset U\cap X_h$ for each $h\in V$. Put $V^*=\cup\{\{h\}{\times} V_h(a)\mid h\in V\}$.
Let $A_1=\{h\in H_{k-1}\mid A\cap X_h\neq\emptyset\}$. Since the set $X_h$ is closed in $H_k$ for each $h\in H_{k-1}$, the set $A\cap X_h$ is closed in $H_k$ as well. Since $X_h$ is totally $\overline{\kappa}$-normal for each $h\in A_1$ there exists an open subset $W_h$ of $X_h$ such that $A\cap X_h\subset W_h\subset \cl_{X_h}(W_h)\subset U\cap X_h$. Since $X_h$ is regular we can also assume that for each $h\in A_1$, $W_h$ satisfies the following condition: if $(h,a)\notin A\cap X_h$, then $(h,a)\notin \cl_{X_h}(W_h)$. Let $W=\cup\{W_h\mid h\in A_1\}\cup V^*$. One can check that $W$ is an open subset of $H_k$ and $A\subset W\subset \cl_{H_k}(W)\subset U$ witnessing that $H_k$ is totally $\overline{\kappa}$-normal. Hence $H_n$ is totally $\overline{\kappa}$-normal for each $n\in\mathbb{N}$.

To see that $H_{a,b}(X)$ is totally $\overline{\kappa}$-normal fix any closed subset $B\subset H_{a,b}(X)$ of cardinality $\leq \kappa$ and an open subset $U\subset H_{a,b}(X)$ such that $B\subset U$. By Remark~\ref{rem}, for each $n\in \mathbb{N}$ the set $B_n=B\cap H_n$ is closed in $H_n$. Next we shall construct a sequence of sets $\{V_n\}_{n\in\mathbb{N}}$ which satisfies the following conditions:
\begin{itemize}
\item $B_n\subset V_n\subset \cl_{H_n}(V_n)\subset H_n\cap U$ for each $n\in\mathbb{N}$;
\item $V_n$ is an open subset of $H_n$ for each $n\in\mathbb{N}$;
\item $V_i\subset V_j$ for each $i\leq j$.
\end{itemize}
Observe that if we construct such a sequence $\{V_n\}_{n\in\mathbb{N}}$, then the set $V=\cup_{n\in\mathbb{N}}V_n$ will be a desired open set in $H_{a,b}(X)$ which contains $B$ and $\cl_{H_{a,b}(X)}(V)\subset U$, witnessing that $H_{a,b}(X)$ is totally $\overline{\kappa}$-normal.

We shall construct the mentioned above family $\{V_n\}_{n\in\mathbb{N}}$ by the induction.
Since $H_1$ is totally $\overline{\kappa}$-normal there exists an open set $W\subset H_1$ such that $B_1\subset W\subset \cl_{H_1}(W)\subset U\cap H_1$. Put $V_1=W$. Assume that we already constructed sets $V_1,\ldots, V_n$. Since $U\cap H_{n+1}$ is an open subset of $H_{n+1}$ and $V_n\subset U\cap H_{n+1}$, the definition of the topology on $H_{n+1}$ and the regularity of $X$ imply that for each $h\in V_n$ there exists an open neighborhood $V_h$ of $a$ in $X$ such that $\{h\}{\times}\cl_{X}(V_h)\subset U$. Let $V_n^{'}=\cup\{V_h\mid h\in V_n\}$. Obviously, $V_n\subset V_n^{'}$ and $V_n^{'}$ is an open subset of $H_{n+1}$. Put $C_{n+1}=B_{n+1}\setminus V_n^{'}$. Recall that $H_{n+1}$ is totally $\overline{\kappa}$-normal. Hence for a closed set $C_{n+1}$ of cardinality $\leq \kappa$ there exists an open subset $W_{n+1}\subset H_{n+1}$ such that $C_{n+1}\subset W_{n+1}\subset \cl_{H_{n+1}}(W_{n+1})\subset U\cap H_{n+1}$. Finally, put $V_{n+1}=V_n^{'}\cup W_{n+1}$. At this point it is easy to check that so defined family $\{V_n\}_{n\in\mathbb{N}}$ satisfies the above conditions.
\end{proof}

\section{Main result}
For any ordinals $\alpha,\beta$ by $[\alpha,\beta)$ ($[\alpha,\beta]$, resp.) we denote the set of all
ordinals $\gamma$ such that $\alpha\leq \gamma <\beta$ ($\alpha\leq \gamma\leq \beta$, resp.).
Further if some ordinal $\alpha$ is considered as a topological space, then $\alpha$ is assumed to carry the order topology.

Before formulating the main result of this paper we prove a few auxiliaries lemmas.

\begin{lemma}\label{l0}
Let $\kappa$ be a cardinal, $\xi$ be a regular cardinal $>\kappa^+$ and $Y$ be a $T_1$ space of pseudo-character $\psi(Y)\leq \kappa$. Then for each continuous map $f:\xi\rightarrow Y$ there exist $y\in Y$ and $\mu\in \xi$ such that $[\mu,\xi)\subset f^{-1}(y)$.
\end{lemma}

\begin{proof}
Let $B=\{y\in Y\mid f^{-1}(y)$ is a cofinal subset of $\xi\}$.
We claim that $|B|\leq 1$. Indeed, if there exist distinct elements $y_1,y_2\in B$, then the sets $f^{-1}(y_1)$ and $f^{-1}(y_2)$ are closed and unbounded in $\xi$ providing that $f^{-1}(y_1)\cap f^{-1}(y_2)\neq \emptyset$ which yields a contradiction.
Hence there are two cases to consider:
\begin{itemize}
\item [1)] the set $B$ is empty;
\item [2)] the set $B$ is singleton;
\end{itemize}

Consider case 1. Put $\alpha_0=0$ and for each $\beta\leq \kappa^+$ put $\alpha_{\beta}=\sup(f^{-1}(f(\alpha_{\beta-1})))+1$ if $\beta$ is a successor ordinal and $\alpha_{\beta}=\sup\{\alpha_{\delta}\mid \delta<\beta\}$ if $\beta$ is a limit ordinal. Since $cf(\xi)>\kappa^+$ the sequence $\{\alpha_{\beta}\mid \beta\leq \kappa^+\}$ is a subset of $\xi$. Observe that for each distinct $\beta_1,\beta_2\in\kappa^++1$, $f(\alpha_{\beta_1})\neq f(\alpha_{\beta_2})$. Put $y=f(\alpha_{\kappa^+})$ and observe that $cf(\alpha_{\kappa^+})=\kappa^+$. Since $\psi(Y)\leq \kappa$ there exists a family $\mathcal{U}$ of cardinality $\leq \kappa$ of open neighborhoods of $y$ such that $\cap \mathcal{U}=\{y\}$. Then for each $U\in \mathcal{U}$, $f^{-1}(U)$ is an open neighborhood of $\alpha_{\kappa^+}$ providing that there exists $\mu_U\in \alpha_{\kappa^+}$ such that $[\mu_U,\alpha_{\kappa^{+}}]\subset f^{-1}(U)$. Since $cf(\alpha_{\kappa^+})>\kappa$ the ordinal $\mu=\sup\{\mu_U\mid U\in \mathcal{U}\}$ belongs to $\alpha_{\kappa^+}$. Hence $[\mu,\alpha_{\kappa^+}]\subset \cap \{f^{-1}(U)\mid U\in\mathcal{U}\}=f^{-1}(y)$. Then for each ordinal $\alpha_{\beta}\in [\mu,\alpha_{\kappa^+}]$, $f(\alpha_{\beta})=f(\alpha_{\kappa^+})$ which implies a contradiction. Hence case 1 is not possible.

Consider case 2. Let $y\in B$. Observe that $f^{-1}(y)$ is closed and unbounded in $\xi$. Fix any family $\mathcal{U}$ of cardinality $\leq \kappa$ of open neighborhoods of $y$ such that $\cap \mathcal{U}=\{y\}$. Since for each $U\in \mathcal{U}$, $f^{-1}(U)$ is an open set which contains $f^{-1}(y)$, for each $\alpha\in f^{-1}(y)$ there exists an ordinal $\mu^U_{\alpha}\in \alpha$ such that $[\mu^U_{\alpha},\alpha]\subset f^{-1}(U)$. For each $U\in\mathcal{U}$ define the map $h_U:f^{-1}(y)\rightarrow \xi$ by $h_U(\alpha)=\mu^U_{\alpha}$, $\alpha\in f^{-1}(y)$. Since $\xi$ is a regular cardinal, for each $U\in\mathcal{U}$ the map $h_U$ satisfies conditions of Fodor's Lemma~\cite[Lemma III.6.14]{Kun}. Hence for each $U\in\mathcal{U}$ there exists an ordinal $\mu_U\in \xi$ and an unbounded (even stationary) in $\xi$ subset $A\subset f^{-1}(y)$ such that $h_U(\alpha)=\mu_U$ for each $\alpha\in A$. At this point it is clear that $[\mu_U,\xi)\subset f^{-1}(U)$ for any $U\in \mathcal{U}$. Since $cf(\xi)>\kappa$ the ordinal $\mu=\sup\{\mu_U\mid U\in\mathcal{U}\}$ belongs to $\xi$. Hence $[\mu,\xi)\subset \cap\{f^{-1}(U)\mid U\in \mathcal{U}\}=f^{-1}(y)$.
\end{proof}


For any distinct cardinals $\alpha,\beta$, by $T_{\alpha,\beta}$ we denote the punctured Tychonoff $(\alpha,\beta)$-plank, i.e., the subspace $(\alpha+1){\times}(\beta+1)\setminus\{(\alpha,\beta)\}$ of the Tychonoff product $(\alpha+1){\times}(\beta+1)$.

\begin{lemma}\label{l1}
Let $\kappa$ be a cardinal and $Y$ be a $T_1$ space such that $\psi(Y)\leq \kappa$. Then for each regular cardinals $\lambda, \xi$ such that $\ka^+<\lambda<\xi$ and for each continuous map $f:T_{\lambda,\xi}\rightarrow Y$ there exist $y\in Y$, $\alpha\in\lambda$ and $\mu\in \xi$ such that $[\alpha,\lambda]{\times}[\mu,\xi]\setminus\{(\lambda,\xi)\}\subset f^{-1}(y)$.
\end{lemma}

\begin{proof}
By Lemma~\ref{l0}, there exist $y\in Y$ and $\alpha\in\lambda$ such that $[\alpha,\lambda){\times}\{\xi\}\subset f^{-1}(y)$.
Since $\psi(Y)\leq \kappa$ there exists a family $\mathcal{U}$ of cardinality $\leq \kappa$ of open neighborhoods of $y$ such that $\cap \mathcal{U}=\{y\}$. The continuity of $f$ implies that for each $U\in \mathcal{U}$ the set $f^{-1}(U)$ is open and contains the set $[\alpha,\lambda){\times}\{\xi\}$. Since $\lambda<cf(\xi)$ for each $U\in\mathcal{U}$ there exists $\mu_U\in \xi$ such that $[\alpha,\lambda){\times}[\mu_U,\xi]\subset f^{-1}(U)$. Since $\kappa<cf(\xi)$ the ordinal $\mu=\sup\{\mu_U\mid U\in \mathcal{U}\}$ belongs to $\xi$. Then $[\alpha,\lambda){\times} [\mu,\xi]\subset \cap\{f^{-1}(U)\mid U\in \mathcal{U}\}=f^{-1}(y)$. Since the set $\{y\}$ is closed in $Y$
$$[\alpha,\lambda]{\times}[\mu,\xi]\setminus\{(\lambda,\xi)\}\subset \overline{[\alpha,\lambda){\times} [\mu,\xi]}\subset \overline{f^{-1}(y)}=f^{-1}(y).$$
\end{proof}

\begin{lemma}\label{l2}
Let $\kappa$ be a cardinal and $\lambda$ be any ordinal $\geq \kappa$. Then $\lambda$ if $\kappa$-accumulative.
\end{lemma}

\begin{proof}
Fix any subset $X\subset \lambda$ such that $|X|\leq \kappa$. Observe that the map $f:\overline{X}\setminus (X\cup \{\sup X\})\rightarrow X$ defined by $f(\alpha)=\min (X\setminus (\alpha+1))$ is injective. Hence $|\overline{X}|\leq\kappa$ witnessing that $\lambda$ is $\kappa$-accumulative.
\end{proof}

Since the finite Tychonoff product of $\kappa$-accumulative spaces is $\kappa$-accumulative and a subspace of a $\kappa$-accumulative space is $\kappa$-accumulative, Lemma~\ref{l2} implies the following:

\begin{corollary}\label{ck}
For any cardinals $\alpha,\beta, \kappa$ such that $\min\{\alpha,\beta\}\geq \kappa$ the space $T_{\alpha,\beta}$ is $\kappa$-accumulative.
\end{corollary}



The following Theorem resolves Problem~\ref{p1} and Problem~\ref{p2} and is the main result of this paper.

\begin{theorem}\label{main}
For each cardinal $\ka$ there exists a regular infinite $\ka$-bounded space $R_{\kappa}$ such that for any $T_1$ space $Y$ of pseudo-character $\psi(Y)\leq \kappa$ each continuous map $f:R_{\kappa}\rightarrow Y$ is constant.
\end{theorem}
\begin{proof}
Fix any cardinal $\kappa$. By $T$ we denote the punctured Tychonoff plank $T_{\kappa^{++},\kappa^{+++}}$. Observe that $T$ is $\kappa$-bounded and $\kappa$-accumulative. Let $\mathbb{Z}$ be a discrete set of integers and $a,b$ be distinct points which do not belong to $T{\times}\mathbb Z$. By $R$ we denote the set $T{\times}\mathbb Z\cup\{a,b\}$ endowed with the topology $\tau$ which satisfies the following conditions:
\begin{itemize}
\item the Tychonoff product $T{\times}\mathbb Z$ is open in $R$;
\item if $a\in U\in\tau$, then there exists $n\in\mathbb N$ such that $\{(t,-k)\mid t\in T$ and $k>n\}\subset U$;
\item if $b\in U\in\tau$, then there exists $n\in\mathbb N$ such that $\{(t,k)\mid t\in T$ and $k>n\}\subset U$.
\end{itemize}
One can check that the space $R$ is regular, $\kappa$-bounded and $\kappa$-accumulative (see Corollary~\ref{ck}). For convenience, we denote the subset $T{\times}\{n\}\subset R$ by $T_n$ for each $n\in \mathbb Z$.

On the space $R$ consider the smallest equivalence relation $\sim$ such that
$$(x,\kappa^{+++},2n)\sim (x,\kappa^{+++},2n+1)\hbox{ and } (\kappa^{++},y,2n)\sim(\kappa^{++},y,2n-1)$$
for any $n\in \mathbb Z$, $x\in\kappa^{++}$ and $y\in\kappa^{+++}$.

Let $X$ be the quotient space $R/\sim$ of $R$ by the equivalence relation $\sim$. Being the continuous image of the $\kappa$-bounded space $R$ the space $X$ is $\kappa$-bounded as well. It is straightforward to check that $X$ is regular and hence, by Proposition~\ref{BBR1}, $X$ is totally $\overline{\kappa}$-normal. Also, one can check that $X$ is $\kappa$-accumulative. We claim that for each $T_1$ space $Y$ of pseudo-character $\psi(Y)\leq \kappa$ and for each continuous map $f:X\rightarrow Y$, $f(a)=f(b)$.
Indeed, fix any space $Y$ such that $\psi(Y)\leq \kappa$ and let $f:X\rightarrow Y$ be any continuous map. Observe that for any $n\in \mathbb Z$ the subspace $[T_n]=\{[x]\mid x\in T_n\}\subset X$ is homeomorphic to $T_n$. By Lemma~\ref{l1}, for each $n\in \mathbb Z$ there exist $y_n\in Y$ and ordinals $\alpha_n\in\kappa^{++}$, $\beta_n\in \kappa^{+++}$ such that
$$\{[(x,\kappa^{+++},n)]\mid x\in [\alpha_n,\kappa^{++})\}\subset f^{-1}(y_n) \hbox{ and } \{[(\kappa^{++},y,n)]\mid y\in [\beta_n,\kappa^{+++})\}\subset f^{-1}(y_n).$$
Recall that $[(x,\kappa^{+++},2n)]=[(x,\kappa^{+++},2n+1)]$ and $[(\kappa^{++},y,2n)]=[(\kappa^{++},y,2n-1)]$ for any $n\in \mathbb Z$, $x\in\kappa^{++}$ and $y\in\kappa^{+++}$ which implies that $y_{2n-1}=y_{2n}=y_{2n+1}$, $n\in \mathbb Z$. Put $\alpha=\sup\{\alpha_n\mid n\in \mathbb Z\}$ and $\beta=\sup\{\beta_n\mid n\in \mathbb Z\}$. Hence there exists a unique $y\in Y$ such that
$$D=\{[(x,\kappa^{+++},n)]\mid x\in [\alpha,\kappa^{++}), n\in\mathbb{Z}\}\cup \{[(\kappa^{++},y,n)]\mid y\in [\beta,\kappa^{+++}),n\in\mathbb{Z}\}\subset f^{-1}(y).$$
Observe that $\{a,b\}\subset \overline{D}\subset \overline{f^{-1}(y)}=f^{-1}(y)$. Hence $f(a)=f(b)$.

Next consider the Herrlich extension $H_{a,b}(X)$ of $X$. Recall that for any space $Y$ of pseudo-character $\psi(Y)\leq \kappa$ each continuous function $f:H_{a,b}(X)\rightarrow Y$ is constant. Since $X$ is totally $\overline{\kappa}$-normal and $\kappa$-accumulative, Proposition~\ref{prH} implies that $H_{a,b}(X)$ is totally $\overline{\kappa}$-normal.

Finally, let $R_{\kappa}$ be the Wallman $\kappa$-bounded extension $W_{\kappa}(H_{a,b}(X))$ of $H_{a,b}(X)$. Recall that $H_{a,b}(X)$ is a dense subspace of $R_{\kappa}$. Hence for any $T_1$ space $Y$ of pseudo-character $\psi(Y)\leq \kappa$ each continuous function $f:R_{\kappa}\rightarrow Y$ is constant on the dense subspace $H_{a,b}(X)\subset R_{\kappa}$ witnessing that $f$ is constant on the whole $R_{\kappa}$. By Proposition~\ref{BBR}, the space $R_{\kappa}$ is $\kappa$-bounded. Proposition~\ref{Wk} implies that $R_{\kappa}$ is regular.
\end{proof}

The next theorem shows that the main result never holds if the space $Y$ is not $T_1$.

\begin{theorem}\label{main1} Let $X$ be a non-anti-discrete space and $Y$ be a space which is not $T_1$.
Then there exists a continuous non-constant
map $f:X\rightarrow Y$.
\end{theorem}

\begin{proof}
Since the space $Y$ is not $T_1$ it contains points $p_1$
and $p_2$ such that $p_2\in \overline{\{p_1\}}$. Since $X$ is not anti-discrete it contains a proper open subset $U$. For each $a\in U$ put $f(a)=p_1$ and
for each $b\in X\setminus U$ put $f(b)=p_2$. It is clear that the map $f$ is continuous
and $f(a)\neq f(b)$ for any $a\in U$ and $b\in X\setminus U$.
\end{proof}

Theorem~\ref{main} and Theorem~\ref{main1} provide the following analogue of Theorem~\ref{Her}.

\begin{theorem}
A space $Y$ is $T_1$ if and only if for any cardinal $\kappa$ there exists an infinite regular $\kappa$-bounded space $R_{\kappa}$ such that any continuous map $f:R_k\rightarrow Y$ is constant.
\end{theorem}

\section*{Acknowledgements}
The authors acknowledge Lyubomyr Zdomskyy for his fruitful comments and suggestions. The work reported here
was carried out during the visit of the second named author to the
KGRC in Vienna. He wishes to thank his colleagues in Vienna.

\end{document}